\documentclass{birkjour}
\usepackage{amsmath,amsfonts,latexsym,amscd,color,amssymb,mathrsfs}
\usepackage{graphicx}
\usepackage{float}
\usepackage{enumerate}
\newtheorem{theorem}{Theorem}

\newtheorem{corollary}[theorem]{Corollary}

\newtheorem{remark}[theorem]{Remark}
\newtheorem{example}[theorem]{Example}

\DeclareMathOperator{\ii}{{i}}

\newcommand{\Comp}{{\mathbb{C} }}

\begin{document}
\title[Global properties of eigenvalues of  perturbations of matrices.]{ Global properties of eigenvalues of parametric rank one perturbations for unstructured and structured matrices II. }

\author[A.C.M. Ran]{Andr\'e C.M. Ran}
\address{Afdeling Wiskunde, Faculteit
    der Exacte Wetenschappen, Vrije Universiteit Amsterdam, De Boelelaan
    1111, 1081 HV Amsterdam, The Netherlands
and Research Focus: Pure and Applied Analytics, North West University, South Africa. \\ ORCID: 0000-0001-9868-8605. }
   \email{a.c.m.ran@vu.nl}. 
\author[M. Wojtylak]{Micha\l{} Wojtylak}
\address{Wydzia\l{} Matematyki i Informatyki, Uniwersytet Jagiello\'nski, 
ul. \L ojasiewicza 6, 30-348 Krak\'ow, Poland.\\
ORCID:  0000-0001-8652-390X}
\email{michal.wojtylak@uj.edu.pl}


\begin{abstract}
We present in this note a correction to Theorem 17 in \cite{WR} and sharpen the estimates for eigenvalues of parametric rank one perturbations given in that theorem.
\end{abstract}

\maketitle

\section{Introduction and Preliminaries} 

This note concerns an erratum and addendum to \cite{WR}, in particular to  Theorem 17.
One of the main points of the theorem is to show the asymptotic behavior of the eigenvalues of $B(\tau)=A+\tau uv^*$ when $|\tau|\to\infty$. The statement is that these eigenvalues trace out a set of curves which asymptotically approximates a set of non-intersecting circles. The statement of the theorem needs a small correction. The proof  contains some independent miscalculations, which fortunately, 
had no effect on the validity of the corresponding statements. Therefore, we find it necessary to present a fully corrected version with a new proof. The detailed list of corrections can be found in Section~\ref{s:core}.
The results are illustrated by an example.

First we repeat the setting of \cite{WR}, to make this erratum and addendum independently readable.
Let $A$ be an $n\times n$ complex matrix, and let $u,v$ be two vectors in $\mathbb{C}^n$. We consider the eigenvalues of the parametric rank one perturbation $B(\tau)=A+\tau uv^*$  of $A$. Denote by $m_A(\lambda)$ the minimal polynomial of $A$, and define
$$
p_{uv}(\lambda)=v^*m_A(\lambda)(\lambda I_n-A)^{-1} u.
$$
Observe that this is a polynomial (see \cite{WR1}), and 
\begin{align}\label{detBA}
\det(\lambda I_n-B(\tau)) &= \det(\lambda I_n -A)\cdot (1-\tau v^* (\lambda I_n-A)^{-1} u) \nonumber \\
&= \frac{\det(\lambda I_n -A)}{m_A(\lambda)} (m_A(\lambda)-\tau p_{uv}(\lambda)).
\end{align}
Also introduce $Q(\lambda)=v^*(\lambda I_n-A)^{-1}u$. By Proposition 2 in \cite{WR}, if $\lambda_0$ is not an eigenvalue of $A$, then it is an eigenvalue of $B(\tau_0)$ of multiplicity $\kappa\geq 1$,  if and only if 
$\tau_0Q(\lambda_0)=1, Q^\prime(\lambda_0)=0, \ldots , Q^{(\kappa-1)}(\lambda_0)=0, Q^{(\kappa)}(\lambda_0)\not= 0$. In this case $\lambda_0$ has geometric multiplicity one.

\section{Main results}

We are interested in the behavior of the eigenvalues of $B(\tau)$, where $\tau=te^{\ii\theta}$ as functions of $\theta$ for fixed $t$, and then in particular in what happens as $t\to\infty$.

For that reason, introduce for $t>0$ the set
$$
\sigma(A,u,v;t)= \bigcup_{0\leq\theta\leq 2\pi} \sigma(A+te^{\ii\theta}uv^*)\setminus\sigma(A).
$$
Theorem 17 in \cite{WR} describes the asymptotic behavior of (parts of) these sets as $t\to\infty$. 
We correct, complete and extend the result.

\begin{theorem}\label{vAu}
Let $A\in\Comp^{n\times n}$, $u,v\in\Comp^n$ and let $l\in\mathbb{N}$ denote the degree of the minimal polynomial $m_A(\lambda)$. 
If 
\begin{equation}\label{vAue0}
v^*u=\cdots =v^*A^{l-1}u=0
\end{equation}
then $p_{uv}(\lambda)\equiv0$ and $\sigma(A+\tau uv^*)=\sigma(A)$ for any $\tau\in\Comp$.
If
\begin{equation}\label{vAue}
v^*u=\cdots =v^*A^{\kappa-1}u=0,\quad v^*A^{\kappa}u\not= 0,
\end{equation}
for some $\kappa \in \{0,\dots,l-1 \}$ then the following statements hold.
\begin{enumerate}[\rm (i)]
\item\label{new-i} $p_{uv}(\lambda)$ is of degree $l-\kappa-1$;
\item\label{new-ii} $l-\kappa-1$ eigenvalues  of $B(\tau)$ converge to the roots of $p_{uv}(\lambda)$ as
$\tau\to\infty$; 
\item\label{new-iii} there are $\kappa+1$ eigenvalues $\lambda_1(\tau),\dots,\lambda_{\kappa+1}(\tau)$ of $A+\tau uv^*$ having the following Puiseux expansion at $\tau=\infty$
\begin{equation}\label{eq:lambdatau1a}
\lambda_j(\tau)=c_{-1}\tau^{\frac{1}{\kappa+1}}+ c_0+ c_1\tau^{-\frac{1}{\kappa+1}}+\cdots,\quad j=1,\dots \kappa+1, 
\end{equation}
where
\begin{align*}
c_{-1}&=(v^*A^\kappa u)^{\frac{1}{\kappa+1}},\\
c_0&=\frac{1}{\kappa+1}\cdot\frac{v^*A^{\kappa+1}u}{v^*A^\kappa u},\\
c_1&=
\frac{1}{\kappa+1}\cdot\frac{1}{(v^*A^\kappa u)^{1+\frac{1}{\kappa+1}}}\cdot
\left( v^*A^{\kappa+2}u-\frac{\kappa+2}{2(\kappa+1)}\cdot\frac{(v^*A^{\kappa+1}u)^2}{v^*A^\kappa u}\right);
\end{align*}


\item\label{new-iv} if $\zeta$ is a root of the polynomial $p_{uv}(\lambda)$ of multiplicity  $k$ and is not a root of $m_A(\lambda)$, then there are $k$ eigenvalues of $A+\tau uv^*$ converging to $\zeta$ with $\tau\to\infty$ having the following Puiseux expansion at $\tau=\infty$
\begin{equation}\label{eq:lambdatau1a-next}
\lambda_j(\tau) =\zeta-b_{1}\tau^{-\frac{1}{k}}-b_{2}\tau^{-\frac{2}{k}}-b_{3}\tau^{-\frac{3}{k}} - \cdots ,\quad j=1,\dots,k,
\end{equation}
where, using
$a_{m}=a_m(\zeta)=v^*(\zeta I_n-A)^{-m}u$ for $m\geq0$, 
\begin{align*}
b_{1}&=b_1(\zeta)=a_{k+1}^{-\tfrac{1}{k}},\qquad a_{k+1}\not= 0,\\
b_{2}&=b_2(\zeta)=-\frac{1}{k}\cdot \frac{b_1^2\, a_{k+2}}{a_{k+1}}.
\end{align*}
\end{enumerate}
\end{theorem}

The complex roots in equations \eqref{eq:lambdatau1a} and \eqref{eq:lambdatau1a-next} should be understood as in the theory of Puiseux series: each particular root determines uniquely the eigenvalue $\lambda_j$, see also Remark~\ref{remcol} below. 

\begin{remark}\rm Let us comment on some genericity issues appearing in Theorem~\ref{vAu}. It was shown in \cite{WR1} that for generic $u,v$ (i.e., $u,v$ with arbitrary complex entries except an algebraic subset of $\Comp^{2n}$, depending possibly on $A$) we have that $v^*u\neq 0$ and the roots of $p_{uv}(\lambda)$ are all simple and disjoint with the roots of  $m_A(\lambda)$. Hence, in Theorem~\ref{vAu} we have generically $\kappa=1$ and in part \eqref{new-iv} for each root $\zeta$ of $p_{uv}(\lambda)$ we have $k=1$ and $m_A(\zeta)\neq0$. However, in general, many different situations might occur. For example, a root of $p_{uv}(\lambda)$ might be a root of $m_A(\lambda)$, as can be seen in Example~\ref{onemore} below. 

Further, note that the multiplicities of the roots of $p_{uv}(\lambda)$ and the number $\kappa$ are in some relation, e.g., due to Theorem~\ref{vAu}\eqref{new-i}.
In particular, if $\kappa=l-1$ then $p_{uv}(\lambda)$ is a constant, nonzero polynomial and the  only limit point of eigenvalues is infinity. 
There are, however, some other hidden constraints relating $\kappa$ and the multiplicities of $p_{uv}(\lambda)$ and their nature  needs to be studied more intensively in future work. 

All there comments explain the role of the assumptions in Corollary~\ref{cor:curves} below.  
\end{remark}

\begin{corollary}\label{cor:curves}  With the notation of Theorem~\ref{vAu},  if \eqref{vAue} holds and all the roots $\zeta_1,\dots,\zeta_\nu$ of $p_{uv}(\lambda)$ are not roots of  $m_A(\lambda)$, then
$\sigma(A,u,v;t)$  for sufficiently large  $t=|\tau|$  can be parametrized by disjoint curves $\Gamma_1(\theta),\dots,\Gamma_{\nu+1}(\theta)$, where the $\kappa +1$ eigenvalues which go to infinity trace out a curve
$$
\Gamma_{\nu+1}(\theta) =c_{-1} t^{\frac{1}{\kappa+1}}e^{\ii\theta}+c_0+c_1 t^{-\frac{1}{\kappa+1}}e^{-\ii\theta}+ O(t^{-\frac{2}{\kappa+1}}),\quad 0\leq\theta\leq2\pi,
$$
while the $k_j$ eigenvalues near $\zeta_j$ trace out a curve $\Gamma_j(\theta)$ which is of the form
$$
\Gamma_j(\theta)=\zeta_j-b_{1}(\zeta_j) t^{-\frac{1}{k_j}}e^{\ii\theta}
-b_{2}(\zeta_j)t^{-\frac{2}{k_j}}e^{2\ii\theta} +  O(t^{-\frac{3}{k_j}}), \ j\in\{1, \ldots , \nu\}.
$$
 In both cases above the $O$ is with respect to $t=|\tau|\to\infty$ and is uniform in $\theta\in[0,2\pi]$.
\end{corollary}

\begin{proof}[Proof of Theorem~\ref{vAu}] 
Assume first \eqref{vAue0} holds. Expanding  the resolvent at infinity we get
 \begin{align*}
p_{uv}(\lambda)&=m_A(\lambda)v^* (\lambda I_n -A)^{-1}u\\
&=m_A(\lambda)v^*\left( \sum_{j=1}^\infty \lambda ^{-j-1} A^j\right)u\\
&= \sum_{j=l}^\infty \lambda ^{-j-1}m_A(\lambda) v^* A^ju  .
 \end{align*}
However, recall that $p_{uv}(\lambda)$ is a polynomial, hence $v^*A^ju=0$ for all $j=0,1,2,\dots$ and $p_{uv}(\lambda)\equiv0$. In consequence, $\det(\lambda I_n-B(\tau))=\det(\lambda I_n-A)$ for all $\tau\in\Comp$ by \eqref{detBA}.

Assume now that \eqref{vAue} holds.
Statements \eqref{new-i} and \eqref{new-ii} were proved in \cite{WR}, the proof did not contain errors. Let us now show (iii).
 For large values of $|\tau|$ consider the eigenvalues of $A+\tau uv^*$ which are not eigenvalues of $A$. These are among the roots of $m_A(\lambda)-\tau p_{uv}(\lambda)$. Dividing by $\tau$, and viewing $s=1/\tau$ as a variable, they are also roots of $sm_A(\lambda)-p_{uv}(\lambda)$, and hence we have by general theory concerning the behavior of roots of a polynomial under a perturbation such as this that the roots are given by Puiseux series, see e.g., \cite{Knopp}, Part II, Chapter V, and \cite{MP},Theorem 9.1.1. For the large eigenvalues of $B(\tau)$ we can make this more precise as follows. Recall that the eigenvalues of $B(\tau)$ which are not eigenvalues of $A$ satisfy $\tau v^*(\lambda I-A)^{-1}u=1$. For $|\lambda|> \|A\|$ this may be rewritten as
$$
1=\tau v^* \sum_{j=0}^\infty \frac{A^j}{\lambda^{j+1}} u=\tau \sum_{j=\kappa}^\infty \frac{v^*A^ju}{\lambda^{j+1}},
$$
where in the last step we used the definition of $\kappa$. Hence
\begin{equation}\label{eq:lambdatau1}
\frac{\lambda^{\kappa+1}}{\tau} =v^*A^\kappa u+\frac{1}{\lambda}v^*A^{\kappa+1}u+\frac{1}{\lambda^2}v^*A^{\kappa+2}u+\cdots .
\end{equation}
We can write $\lambda$ as a Puiseux series as
\begin{equation}\label{eq:lambdatau1aaaa}
\lambda=c_{-1}\tau^{\frac{1}{\kappa+1}}+ c_0+ c_1\tau^{-\frac{1}{\kappa+1}}+\cdots,
\end{equation}
or equivalently,
\begin{equation}\label{eq:lambdatau2}
\frac{\lambda}{\tau^{\frac{1}{\kappa+1}}} =c_{-1}+c_0\tau^{-\frac{1}{\kappa+1}}+c_1\tau^{-\frac{2}{\kappa+1}}+\cdots .
\end{equation}
Taking the $(\kappa+1)$'th power of this we arrive at
\begin{align}\label{eq:lambdatau3}
\frac{\lambda^{\kappa+1}}{\tau} = &c_{-1}^{\kappa+1} + (\kappa+1)c_{-1}^\kappa c_0
\tau^{-\frac{1}{\kappa+1}} \nonumber \\ +&
\left(  (\kappa+1)c_{-1}^\kappa c_1+\begin{pmatrix} \kappa+1 \\ 2\end{pmatrix} c_{-1}^{\kappa-1}c_0^2\right) \tau^{-\frac{2}{\kappa+1}} + O(\tau^{-\frac{3}{\kappa+1}}).
\end{align}
From the leading terms in \eqref{eq:lambdatau1} and \eqref{eq:lambdatau3} we see that 
\begin{equation}\label{eq:cminone}
c_{-1}^{\kappa+1} =v^*A^\kappa u.
\end{equation}
Now consider $\left(\frac{\lambda^{\kappa+1}}{\tau} -v^*A^\kappa u\right)\lambda$. By \eqref{eq:lambdatau1} this is equal to 
\begin{equation}\label{eq:formulasecondterm}
\left(\frac{\lambda^{\kappa+1}}{\tau} -v^*A^\kappa u\right)\lambda =v^*A^{\kappa+1}u+\frac{1}{\lambda}v^*A^{\kappa+2}u+O(\lambda^{-2}).
\end{equation}
On the other hand, inserting \eqref{eq:cminone} into \eqref{eq:lambdatau3}, and then inserting \eqref{eq:lambdatau1a} we obtain
{\small
\begin{align}\label{eq:formulacnul}
& \left(\frac{\lambda^{\kappa+1}}{\tau} -v^*A^\kappa u\right)\lambda =\nonumber \\
=& \lambda \left( (\kappa+1)c_{-1}^\kappa c_0
\tau^{-\frac{1}{\kappa+1}}  +
\left(  (\kappa+1)c_{-1}^\kappa c_1+\begin{pmatrix} \kappa+1 \\ 2\end{pmatrix} c_{-1}^{\kappa-1}c_0^2\right) \tau^{-\frac{2}{\kappa+1}} + O(\tau^{-\frac{3}{\kappa+1}})\right)
\nonumber \\
=&\left(c_{-1}\tau^{\frac{1}{\kappa+1}}+ c_0+ c_1\tau^{-\frac{1}{\kappa+1}}+O(\tau^{-\frac{2}{\kappa+1}}\right)\nonumber \\
\cdot &\left( (\kappa+1)c_{-1}^\kappa c_0
\tau^{-\frac{1}{\kappa+1}}  
\left(  (\kappa+1)c_{-1}^\kappa c_1+\begin{pmatrix} \kappa+1 \\ 2\end{pmatrix} c_{-1}^{\kappa-1}c_0^2\right) \tau^{-\frac{2}{\kappa+1}} + O(\tau^{-\frac{3}{\kappa+1}})\right)
\nonumber  \\
=& (\kappa+1)c_{-1}^{\kappa+1}c_0+ (\kappa+1)\left( c_{-1}^\kappa c_0^2+
c_{-1}^{\kappa+1}c_1 +\frac{\kappa}{2}c_{-1}^\kappa c_0^2\right)\tau^{-\frac{1}{\kappa+1}}+ O(\tau^{-\frac{2}{\kappa+1}})
\nonumber  \\
=&(\kappa+1)c_{-1}^{\kappa+1}c_0+(\kappa+1) \left( \frac{\kappa+2}{2}c_{-1}^\kappa c_0^2+
c_{-1}^{\kappa+1}c_1\right)\tau^{-\frac{1}{\kappa+1}}+ O(\tau^{-\frac{2}{\kappa+1}})
\end{align}
}

Comparing formulas \eqref{eq:formulasecondterm} and \eqref{eq:formulacnul} we see
\begin{equation}\label{eq:formulacnula}
(\kappa+1)c_{-1}^{\kappa+1}c_0=v^*A^{\kappa+1}u.
\end{equation}
Using \eqref{eq:cminone} we obtain 
\begin{equation}\label{eq:cnulb}
c_0=\frac{1}{\kappa+1} \cdot \frac{v^*A^{\kappa+1}u}{v^*A^\kappa u}.
\end{equation}
In addition, subtracting the constant term in \eqref{eq:formulasecondterm} and then multiplying by $\lambda$ we obtain
\begin{equation}\label{eq:formulathirdterm}
\left( \left( \frac{\lambda^{\kappa+1}}{\tau} -v^*A^\kappa u\right)\lambda - v^*A^{\kappa+1}u\right)\lambda=v^*A^{\kappa+2u}+O(\lambda^{-1}).
\end{equation}
On the other hand, subtracting the constant term in \eqref{eq:formulacnul} and then multiplying by $\lambda$ we obtain, also using \eqref{eq:formulacnula},
\begin{align*}
&\left( \left( \frac{\lambda^{\kappa+1}}{\tau} -v^*A^\kappa u\right)\lambda - v^*A^{\kappa+1}u\right)\lambda\\
=&\lambda\cdot  (\kappa+1)\left( \frac{\kappa+2}2c_{-1}^\kappa c_0^2+
c_{-1}^{\kappa+1}c_1\right)\tau^{-\frac{1}{\kappa+1}}+ O(\tau^{-\frac{2}{\kappa+1}}) .
\end{align*}
Now use again \eqref{eq:lambdatau1a} to see that
\begin{align}\label{eq:formulacone}
&\left( \left( \frac{\lambda^{\kappa+1}}{\tau} -v^*A^\kappa u\right)\lambda - v^*A^{\kappa+1}u\right)\lambda\nonumber \\
=&
\left(c_{-1}\tau^{\frac{1}{\kappa+1}}+ c_0+O(\tau^{-\frac{1}{\kappa+1}})\right)
\nonumber\\ 
  \cdot&\left(  (\kappa+1)\left( \frac{\kappa+2}{2}c_{-1}^\kappa c_0^2+
c_{-1}^{\kappa+1}c_1\right)\tau^{-\frac{1}{\kappa+1}}+ O(\tau^{-\frac{2}{\kappa+1}}) \right)\nonumber\\
=& (\kappa+1)\left (\frac{\kappa+2}{2}c_{-1}^{\kappa+1} c_0^2+
c_{-1}^{\kappa+2}c_1\right) +O(\tau^{-\frac{1}{\kappa+1}}).
\end{align}
Comparing the constant terms in \eqref{eq:formulathirdterm} and \eqref{eq:formulacone} we see that
\begin{equation*}
(\kappa+1)\left(\frac{\kappa+2}{2}c_{-1}^{\kappa+1} c_0^2+
c_{-1}^{\kappa+2}c_1 \right)=v^*A^{\kappa+2}u.
\end{equation*}
Solving this equation for $c_1$ using the formulas \eqref{eq:cminone} and \eqref{eq:cnulb},
one finds after some computation
$$
c_1=\frac{1}{\kappa+1}\cdot\frac{1}{(v^*A^\kappa u)^{1+\frac{1}{\kappa+1}}}\cdot
\left( v^*A^{\kappa+2}u-\frac{\kappa+2}{2(\kappa+1)}\cdot\frac{(v^*A^{\kappa+1}u)^2}{v^*A^\kappa u}\right),
$$
as stated in the theorem.

\eqref{new-iv}
Since $\zeta$ is a root of $p_{uv}(\lambda)$ and by assumption is not a root of $m_A(\lambda)$ we have $v^*(\zeta I_n-A)^{-1}u=0$. Hence,  for $\lambda$ near $\zeta$ we expand 
\begin{align*}
v^*(\lambda I_n-A)^{-1}u & = v^*((\lambda-\zeta)I_n+(\zeta I_n-A))^{-1}u\\ &=
v^*((\lambda-\zeta)(\zeta I_n-A)^{-1}+I_n)^{-1}(\zeta I_n-A)^{-1}u\\
&=v^*\sum_{j=0}^\infty (-1)^j(\lambda-\zeta)^j(\zeta I_n-A)^{-(j+1)}u\\
&=v^*\sum_{j=1}^\infty (\zeta-\lambda)^j(\zeta I_n-A)^{-(j+1)}u.
\end{align*}
Recall  that any eigenvalue of $A+\tau uv^*$ which is not an eigenvalue of $A$ satisfies
$$
\frac{1}{\tau}=v^*(\lambda  I_n-A)^{-1}u.
$$
As  $\zeta$ is a root of $p_{uv}(\lambda)$ with multiplicity $k$, we have
\begin{equation}\label{handy}
\frac{1}{\tau}=(\zeta-\lambda)^{k}a_{k+1} +(\zeta-\lambda)^{k+1}a_{k+2}+\cdots, \qquad a_{k+1}\not= 0.
\end{equation}
 We express now $\lambda$ in a a Puiseux series in $\tau^{-1}$, this is possible because $\lambda$ is a root of $m_A(\lambda)-\tau p_{uv}(\lambda)=0$. 
Since $\zeta$ has multiplicity $k$ we have that
$$
\lambda =\zeta-b_{1}\tau^{-\frac{1}{k}}-b_{2}\tau^{-\frac{2}{k}}-b_{3}\tau^{-\frac{3}{k}} - \cdots 
$$
for some $b_{1}, b_{2}, \ldots$. 
Then $\zeta-\lambda=b_{1}\tau^{-\frac{1}{k}}+b_{2}\tau^{-\frac{2}{k}}+\cdots$, and inserting that in the equation \eqref{handy} we obtain
\begin{align}
\frac{1}{\tau}&= \frac{1}{\tau}b_{1}^{k}a_{k+1} + kb_{1}^{k-1}\tau^{-\frac{k-1}{k}}\cdot b_{2}\tau^{-\frac{2}{k}}a_{k+1} \nonumber
\\& \  + b_{1}^{k+1}\tau^{-\frac{k+1}{k}}a_{k+2} + {\rm smaller\ order\ terms}\nonumber \\
&=\frac{1}{\tau}b_{1}^{k}a_{k+1}
+\tau^{-\frac{k+1}{k}}\left( k b_{1}^{k-1} b_{2}a_{k+1} +  b_{1}^{k+1}a_{k+2}\right) +\cdots . \label{eq:csandas}
\end{align}
Equating terms of equal powers in $\tau$, for the terms $\frac{1}{\tau}$ on the  left and right hand sides gives
$$
b_{1}=a_{k+1}^{-\tfrac{1}{k}}=\left(\frac{1}{v^*(\zeta I_n-A)^{-(k+1)}u}\right)^{\frac{1}{k}}.
$$
The term on the right hand side with power $\tau^{-\frac{k+1}{k}}$ gives
$$
kb_{2}a_{k+1}+b_{1}^2a_{k+2}=0,
$$
i.e., 
$$
b_{2}=-\frac{1}{k}\cdot \frac{b_{1}^2\, a_{k+2}}{a_{k+1}}.
$$
This completes the proof. \end{proof}

Using the formula for $b_{1}$ we can derive an alternative formula for $b_{2}$ completely in terms of $a_{k+1}$ and $a_{k+2}$, which after some computation, and with proper care for the $k$th roots, becomes
$$
b_{2}=-\frac{1}{k}\cdot\frac{a_{k+2}}{a_{k+1}^{1+\tfrac{2}{k}}}
=-\frac{1}{k}\cdot \frac{v^*(\zeta I_n-A)^{-(k+2)}u}{\left(v^*(\zeta I_n-A)^{-(k+1)}u\right)^{1+\tfrac{2}{k}}}.
$$

\begin{remark}\label{remcol}\rm As stated in the Corollary \ref{cor:curves}, the $\kappa+1$ eigenvalues going to infinity together trace out the curve $\Gamma_{\nu+1}(\theta)$. Let us number the eigenvalues so that these are $\lambda_1(\tau), \ldots , \lambda_{\kappa+1}(\tau)$. 
After possibly renumbering these eigenvalues, one derives
from the theory of Puiseux series, see e.g.~\cite{Knopp}, 
$$
\lambda_j(te^{\ii\theta}) = \sqrt[\kappa+1]{t\, r}\, e^{ \ii(\frac{1}{\kappa+1} (\theta+\theta_{0}) +\frac{2j}{\kappa+1}\pi)}+O(1), \qquad j=1, 2, \ldots , \kappa+1,
$$
where $v^*A^\kappa u=r e^{\ii \theta_0}$. 
As $\theta\to 2\pi$ one has that
$$
\lambda_j(te^{\ii\theta})\to \lambda_{j+1}(t), \quad j=1,\dots , \kappa, \quad \lambda_{\kappa +1} (te^{\ii\theta})\to \lambda_1(t).
$$
\end{remark}

A similar statement holds for the $k_j$ eigenvalues near $\zeta_j$ tracing out the curve $\Gamma_j(\theta)$.

\section{List of corrections}\label{s:core}

\begin{itemize}
\item The analysis of the case indicated in formula \eqref{vAue0} above was missing in \cite{WR}. For completeness, we have included it in the current version.
\item The eigenvalues $\zeta_j$ in point (v) of Theorem 17  of  \cite{WR} were not assumed to be disjoint with the roots of $m_A(\lambda)$.  If some $\zeta_j$ is a root of $m_A(\lambda)$ several things might occure, which need an independent work. A reformulation of the Theorem, including that assumption, was necessary. 
\item A more detailed Puiseux expansion for the eigenvalues for $\tau\to\infty$ was given, both in the case of eigenvalues converging to infinity and to a root of $p_{uv}(\lambda)$. The version in \cite{WR} contained only the first term.
\item All results on the set $\sigma(A,u,v,t)$ were moved to a separate Corollary.  This is partially due to the two previous items, and partially due to presentation issues. 
\item On page 17, line 13 in \cite{WR} the formula given there for $c_1$ is wrong. 
\item On page 17, last three lines, and page 18, the first line in \cite{WR} the display formula contains a mistake which has an effect on the remainder of the proof. There is a factor $(-1)^k$ missing in the summation.
\item The formula at the bottom of page 18 in \cite{WR} contains an error. 
\end{itemize}

\section{Examples}

Let us begin with the promised example when  $m_A(\lambda)$ and $p_{uv}(\lambda)$ have a common root.

\begin{example}\label{onemore}\rm
Let
$$
A=\begin{bmatrix} 1 & 0 &0 \\ 0& 0.5 &0 \\0 & 0 & 1.5
\end{bmatrix},\quad u=v=\begin{bmatrix} 0 \\ 1 \\ 1\end{bmatrix}.
$$
Then $\kappa=1$ and $p_{uv}(\lambda)=2(\lambda-1)^2$. Hence, $A+\tau uv^*$ has  one eigenvalue  converging to infinity and two eigenvalues converging to $1$ with $\tau\to\infty$. However, unlike in Theorem~\ref{vAu}\eqref{new-iv}, in the plot of $\sigma(A,u,v,t)$ we do not have a circle around $1$ formed by two eigenvalues. In the current situation one eigenvalue remains at $1$ for all $\tau\in\Comp$, while the other one forms the full circle, see Figure \ref{f:om}. The plot in Figure \ref{f:om} shows these circles for $t=1$, together with the eigenvalues of $B(te^{\ii\theta})$ for $t=1$ and $\theta=\frac{2\pi j}{200}$ for $j=1, 2, \ldots , 200$.
\begin{figure}[h]
   \begin{center}
\includegraphics[height=4cm]{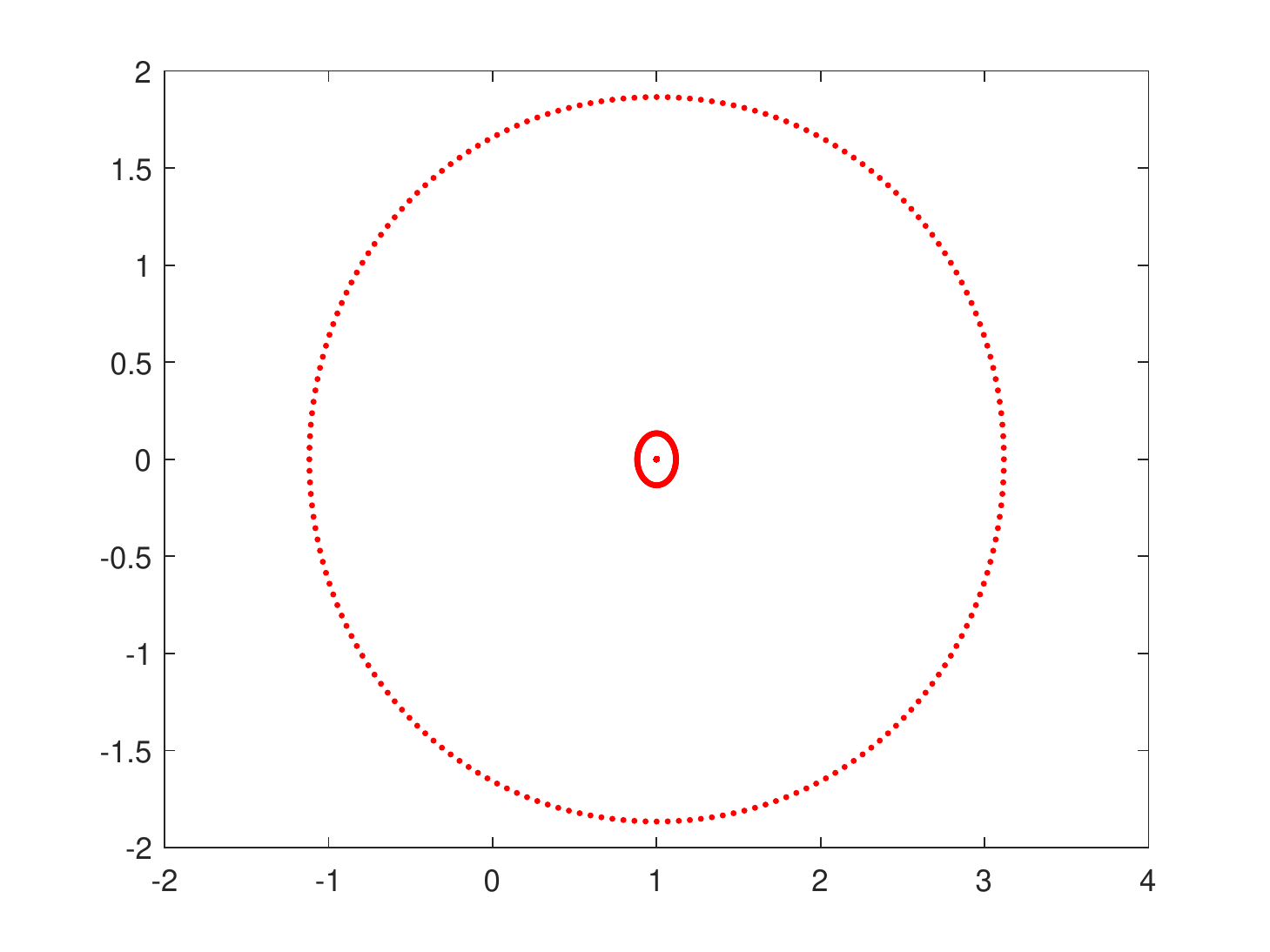}
\end{center}
\caption{Eigenvalues for $t=1$ in Example~\ref{onemore}.}\label{f:om}
\end{figure}
\end{example}

\begin{remark} \rm
The formulas in Theorem~\ref{vAu}\eqref{new-iii}  take an especially nice form in the generic case, when $\kappa=0$. In that case we have
$$
c_{-1}=v^*u, \quad c_0=\frac{v^*Au}{v^*u}, \quad c_{1}=\frac{1}{(v^*u)^2}\left(v^*A^2u -\frac{(v^*Au)^2}{v^*u}\right).
$$
As a first approximation we obtain the circle 
$$
\Gamma(\theta)\approx v^*u \cdot te^{\ii \theta} +\frac{v^*Au}{v^*u},
$$
while a further refinement is the curve
$$
\Gamma(\theta)\approx v^*u \cdot te^{\ii \theta} +\frac{v^*Au}{v^*u} 
+ \frac{1}{(v^*u)^2}\left(v^*A^2u -\frac{(v^*Au)^2}{v^*u}\right) \cdot\frac{1}{t}e^{-\ii\theta}.
$$

Let us also specialize the formulas Theorem~\ref{vAu}\eqref{new-iv}  for the generic case $k=1$. 
In this case we have
$$
b_1=\frac{1}{v^*(\zeta I_n-A)^{-2} u}, \quad b_2= -\frac{v^*(\zeta I_n -A)^{-3}u}{(v^*(\zeta I_n-A)^{-2} u)^3}.
$$
As a first approximation we obtain the circle
$$
\Gamma(\theta)\approx \zeta-b_{1}e^{\ii\theta}\tfrac{1}{t} 
$$
as a second approximation we obtain the curve
$$
\Gamma(\theta)\approx \zeta-b_{1}e^{\ii\theta}\tfrac{1}{t} -b_{2}e^{2\ii\theta}\tfrac{1}{t^2}.
$$
\end{remark}

\begin{example}\rm Consider $A=\begin{bmatrix} -1 & 0 & 0 \\ 0 & 0 & 0 \\ 0 & 0 & 1\end{bmatrix}$, $u=\begin{bmatrix} 1 \\ 1 \\ 1\end{bmatrix}, v=\begin{bmatrix} 1 \\ 2 \\ 3 \end{bmatrix}$. 
Then $p_{uv}(\lambda)= 6\lambda^2+2\lambda-2$, with zeroes $\zeta_1=-\tfrac{1}{6}+\tfrac{\sqrt{13}}{6} \approx 0.4343$ and $\zeta_2=-\tfrac{1}{6}-\tfrac{\sqrt{13}}{6} \approx -0.7676$. 
One computes that $b_{1}(\zeta_1)\approx 0.0489$ and $b_{1}(\zeta_2)\approx 0.0437$. So the first approximation of the curves $\Gamma_j(\theta)$ are given by $ \Gamma_j(\theta) \approx \zeta_j-b_{1}(\zeta_j)e^{\ii \theta}\tfrac{1}{t} $, which are circles with centers at $\zeta_j$ and radii $b_{1}(\zeta_j)$. The plots in Figure \ref{figone} show these circles for $t=1$, together with the eigenvalues of $B(te^{\ii\theta})$ for $t=1$ and $\theta=\frac{2\pi j}{200}$ for $j=1, 2, \ldots , 200$.

\begin{figure}[h]
\begin{center}
\includegraphics[height=3.5cm]{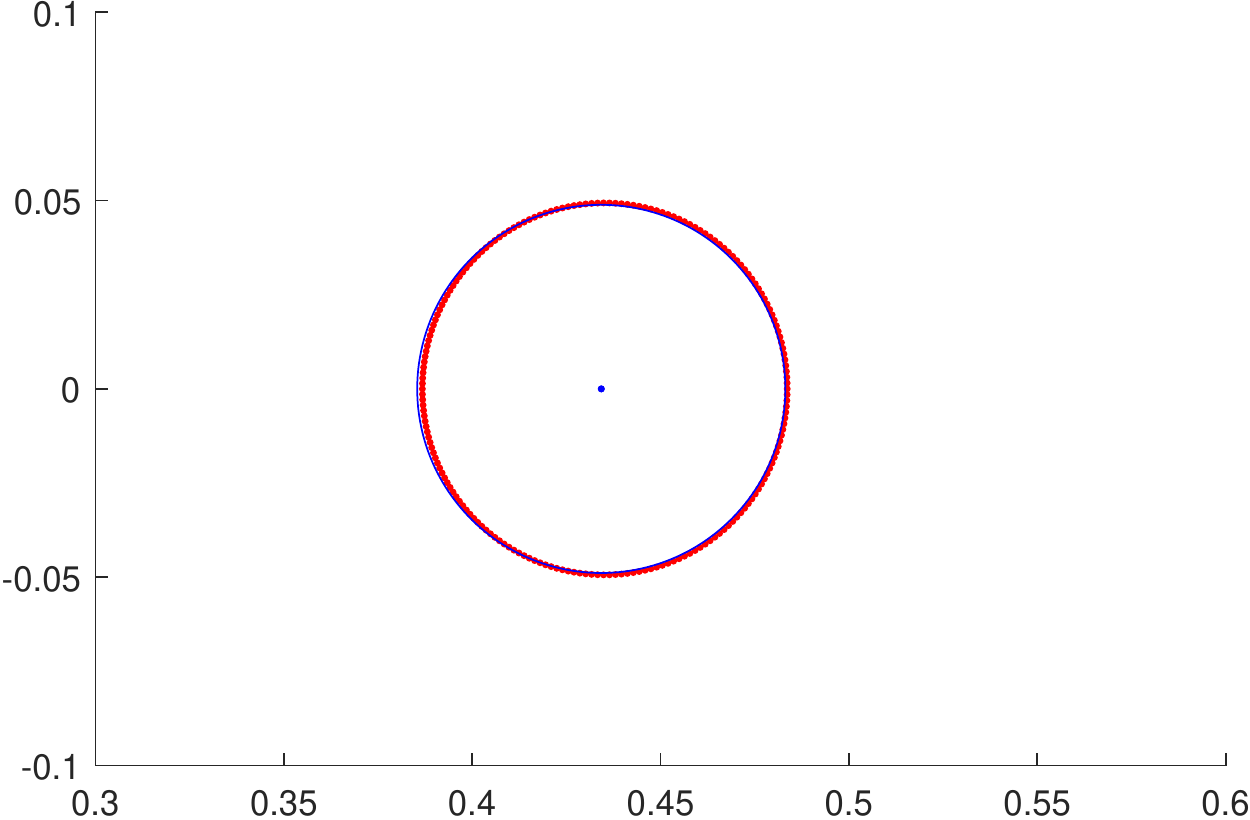}\ \includegraphics[height=3.5cm]{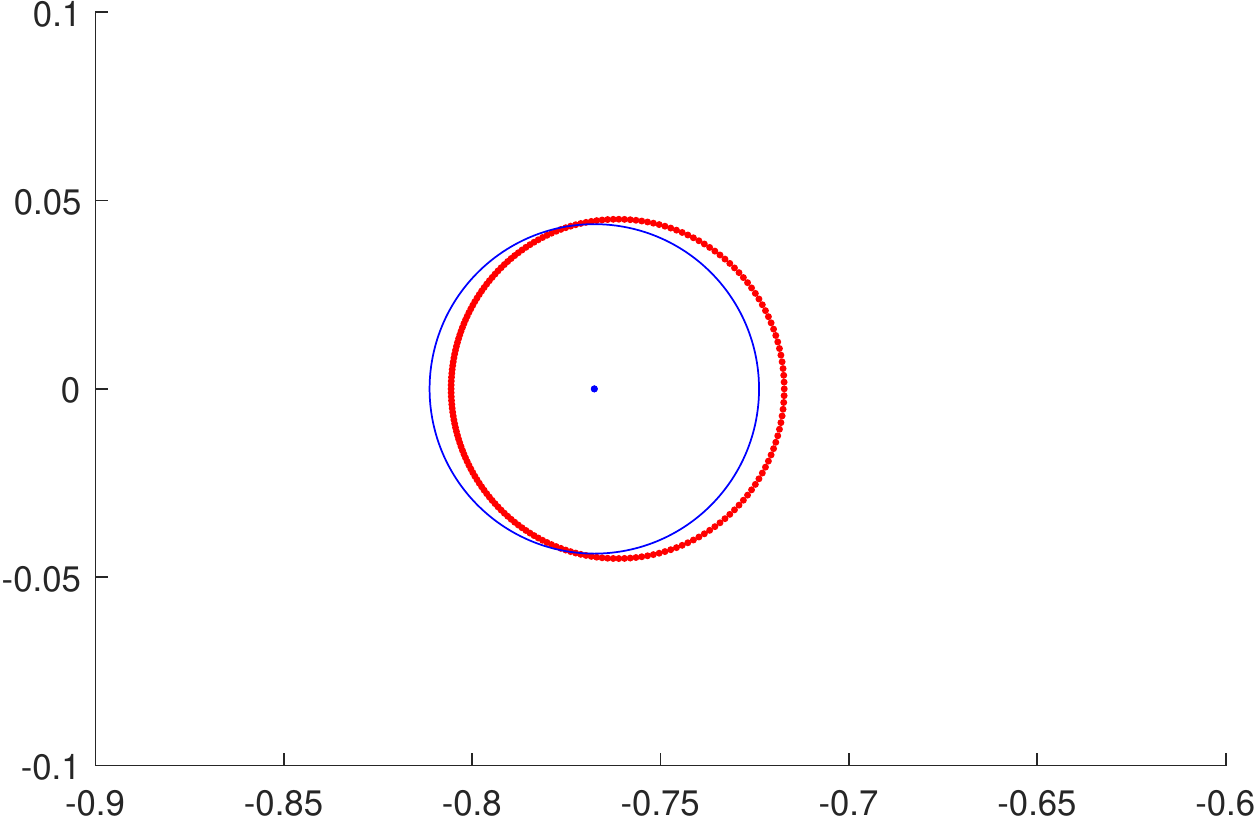}
\end{center}
\caption{Eigenvalues for $t=1$ in red, the circular approximation in blue. On the left around $\zeta_1$, on the right around $\zeta_2$.}\label{figone}
\end{figure}

It is obvious from the graphs that the circle around $\zeta_1$ is already a fairly good approximation for the eigenvalues of $B(\tau)$. However, the circle around $\zeta_2$ is definitely not a very nice approximation. We further compute that $b_{2}(\zeta_1)\approx -9.5594 \times 10^{-4}$ and $b_{2}(\zeta_2)\approx -0.062$. Focusing on the next approximation of $\Gamma_2(\theta)$ we get as an approximation $\Gamma_2(\theta)\approx -0.7676-0.0437e^{\ii\theta}+0.062e^{2\ii\theta}$.
Incorporating this extra term in the approximation leads to a much better approximation as is shown in
Figure \ref{figtwo}. 

\begin{figure}[h]
\begin{center}
\includegraphics[height=4cm]{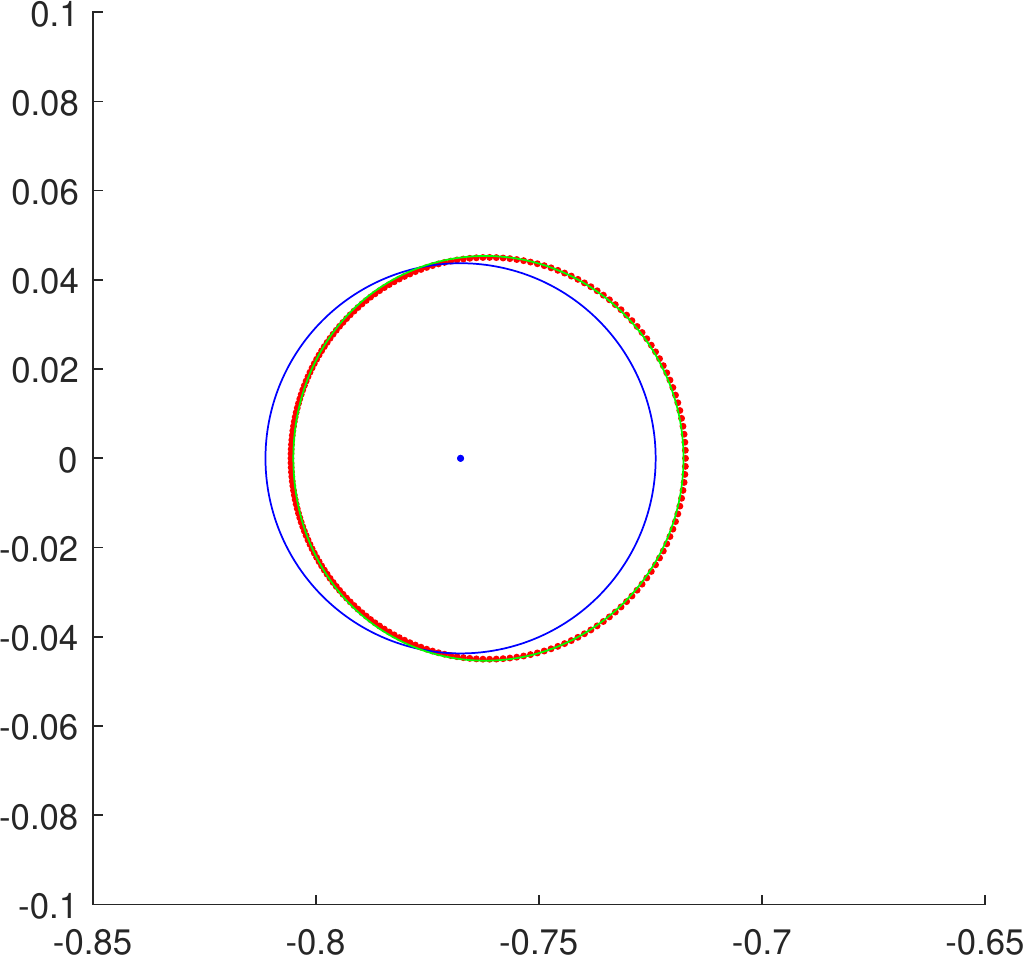}
\end{center}
\caption{Eigenvalues for $t=1$ close to $\zeta_2$ in red, the circular approximation in blue, in green the more accurate approximation with an additional term.}\label{figtwo}
\end{figure}

For the large eigenvalues of $B(\tau)$, already for $\tau=1$ the circular approximation is fairly good, the extra term in the approximation only makes it even better. In this case we have $c_{-1}=6$, $c_0=\frac{1}{3}$ and $c_1=\frac{23}{216}$.
See Figure \ref{figthree}, where for 200 equally spaces values of $\theta$ the eigenvalues of $B(e^{\ii\theta})$ are plotted, together with the circular approximation $\frac{1}{3}+6e^{\ii\theta}$ and the second approximation $\frac{1}{3}+6e^{\ii\theta}+\frac{23}{216}e^{-\ii\theta}$.

\begin{figure}[h]
\begin{center}
\includegraphics[height=4cm]{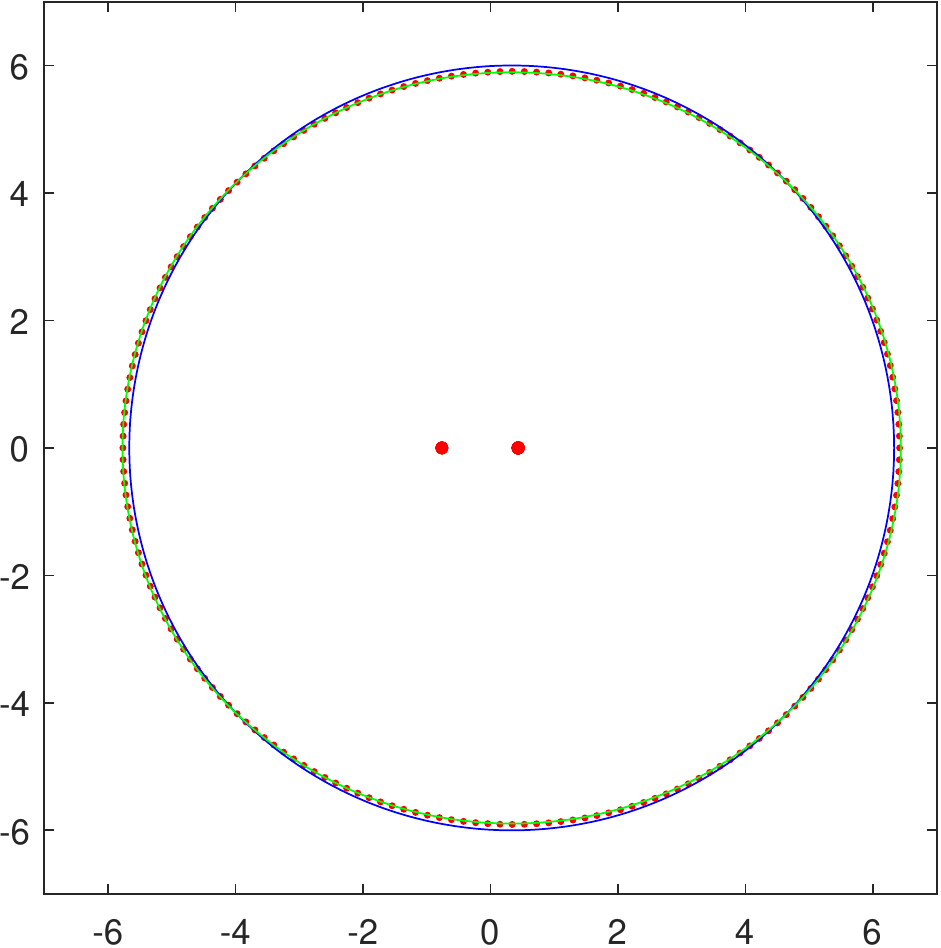}
\end{center}
\caption{The eigenvalues of $B(e^{\ii\theta})$ are plotted in red, the circular approximation is plotted in blue and the second approximation is plotted in green.}\label{figthree}
\end{figure}

For $t=\frac{1}{2}$ the picture is clearer. The approximating circle in this case is 
$\frac{1}{3}+3e^{\ii\theta}$, the second approximation is $\frac{1}{3}+3e^{\ii\theta}+\frac{46}{216}e^{-\ii\theta}$, as is shown in Figure \ref{figfour}

\begin{figure}[h]
\begin{center}
\includegraphics[height=4cm]{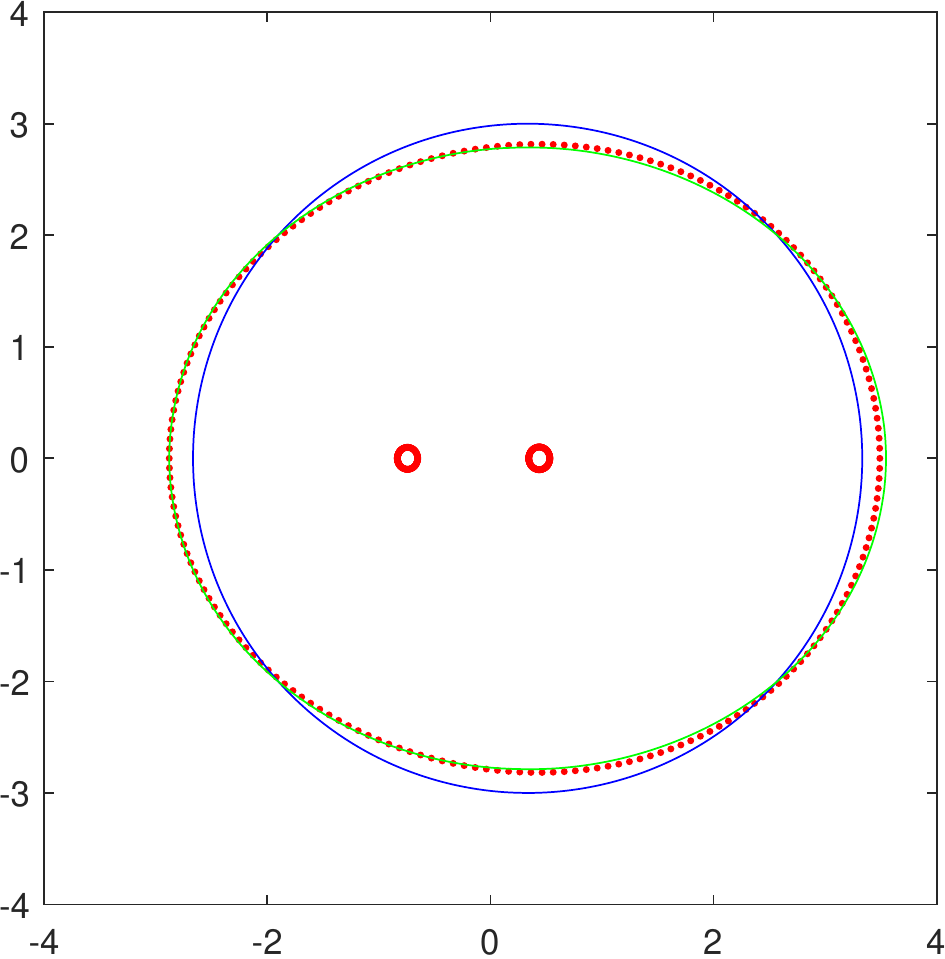}
\end{center}
\caption{The eigenvalues of $B(\tfrac{1}{2}e^{\ii\theta})$ are plotted in red, the circular approximation is plotted in blue and the second approximation is plotted in green.}\label{figfour}
\end{figure}
\end{example}

\begin{example}\rm Take $A$ and $u$ as in the previous example, but let $v=\begin{bmatrix} 1 \\ -\frac{1}{2} \\ -\frac{1}{2}\end{bmatrix}$. Then $v^*u=0$ and $v^*Au=-\frac{3}{2}$, so $\kappa=1$. Furthermore, $v^*A^2u=\frac{1}{2}$ and $v^*A^3u=-\frac{3}{2}$. In that case,
$c_{-1}^2=-\frac{3}{2}$, so $c_{-1}=\sqrt{\frac{3}{2}}\ii$, $c_0=-\frac{1}{6}$ and $c_1=\frac{-5\sqrt{2} }{24\sqrt{3}}\ii$. With $\tau=te^{\ii\theta}$ we obtain for the circular approximation of the two eigenvalues going to infinity
$$
\pm\sqrt{\frac{3}{2}} \sqrt{t}(-\sin(\frac{1}{2}\theta)+\ii\cos(\frac{1}{2}\theta))-\frac{1}{6}.
$$
Adding the extra terms with $c_1$ is a bit more involved; it gives
$$
-\frac{1}{6}\pm \left(\sqrt{\frac{3}{2}} \sqrt{t}(-\sin(\frac{1}{2}\theta)+\ii\cos(\frac{1}{2}\theta))
-\frac{1}{\sqrt{t}}\frac{5\sqrt{2} }{24\sqrt{3}}(\sin(-\frac{1}{2}\theta)-\ii\cos(-\frac{1}{2}\theta))\right).
$$
For $t=4$ Figure \ref{figfive} shows the situation, and also illustrates that the fit for the circle is not satisfactory, while the fit with the next term is essentially better. Obviously, for larger $t$ this will improve even further.

\begin{figure}[h]
\begin{center}
\includegraphics[height=4cm]{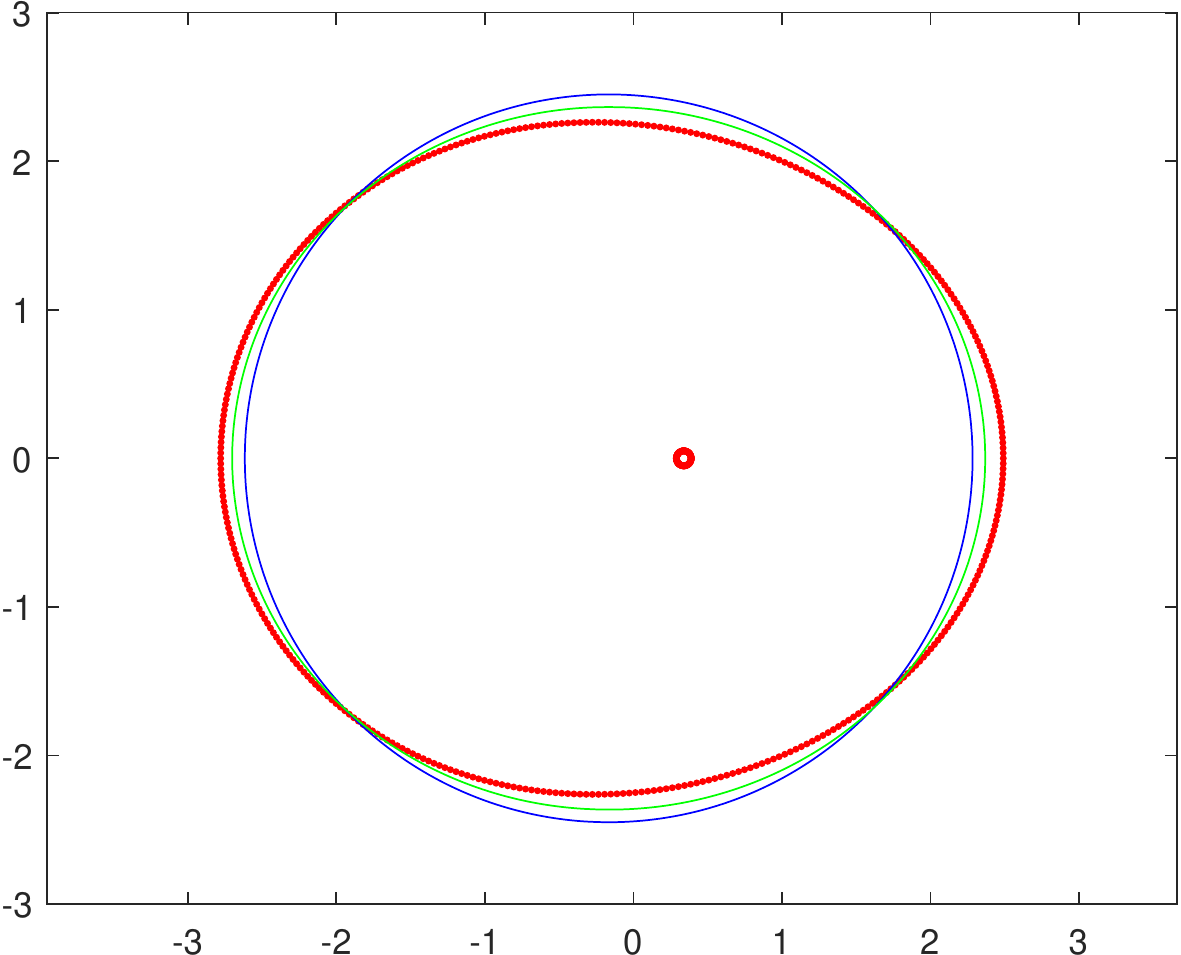}
\end{center}
\caption{The eigenvalues of $B(4e^{\ii\theta})$ plotted in red, the circular approximation in blue and the second approximation in green.}\label{figfive}
\end{figure}
\end{example}

The results of Theorem~\ref{vAu} and Corollary \ref{cor:curves} show how the curves $\Gamma_j$ can be described not only qualitatively, but as the examples show, also quantitatively the results are fairly sharp, certainly if we take into account the correction term to the circular approximation.

\paragraph{\bf Declarations:} $ $\smallskip

\paragraph{\bf Conflict of interest} The authors declare that they have no conflict of interest.\smallskip

\paragraph{\bf Funding}
The research of the second author was funded by the Priority Research Area SciMat under the program
Excellence Initiative – Research University at the Jagiellonian University in Kraków \smallskip

\paragraph{\bf Availability of data and material} Data sharing is not applicable to this article as no datasets were generated or analysed during the current study.\smallskip

\paragraph{\bf Code availability} Not applicable.

\end{document}